\documentclass[centertags,12pt]{amsart}
\usepackage{latexsym}
\usepackage{amsthm}
\usepackage{amssymb}
\usepackage[utf8]{inputenc}
\usepackage{mathrsfs}
\usepackage[english]{babel}
\usepackage{setspace}
\usepackage{ textcomp }
\usepackage{graphicx}
\graphicspath{ {images/} }

\numberwithin{equation}{section}
\makeatletter
\renewcommand{\subsection}{\@startsection
{subsection}{2}{0mm}{\baselineskip}{-0.25cm}
{\normalfont\normalsize\bf}}
\makeatother

\newtheorem{theorem}{Theorem}[section]
\newtheorem{proposition}[theorem]{Proposition}
\newtheorem{lemma}[theorem]{Lemma}
\newtheorem{corollary}[theorem]{Corollary}

   {\theoremstyle{definition}

}
   \theoremstyle{remark}

\newcommand{\F}{{\mathbb F}}
\newcommand{\fq}{{\mathbb F_q}}
\newcommand{\fqs}{{\mathbb F_{q^2}}}
\newcommand{\PGU}{{\rm PGU}}
\newcommand{\PSU}{{\rm PSU}}
\newcommand{\PGL}{{\rm PGL}}
\newcommand{\PG}{{\rm PG}}
\newcommand{\cX}{{\mathcal X}}

\newcommand{\cH}{{\mathcal H}}
\newcommand{\cR}{{\mathcal R}}
\newcommand{\cS}{{\mathcal S}}
\newcommand{\aut}{{\rm Aut}}
\newcommand{\ord}{{\rm ord}}
\newcommand{\diag}{{\rm diag}}
\newcommand{\la}{{\langle}}
\newcommand{\ra}{{\rangle}}
\newcommand{\T}{{\times}}
\newcommand{\RT}{{\rtimes}}

\begin{document}

\author[ M. Montanucci - G. Zini]{Maria Montanucci and Giovanni Zini}
\title{Some Ree and Suzuki curves are not Galois covered by the Hermitian curve
}

\begin{abstract}
The Deligne-Lusztig curves associated to the algebraic groups of type $^2A_2$, $^2B_2$, and $^2G_2$ are classical examples of maximal curves over finite fields. The Hermitian curve $\mathcal H_q$ is maximal over $\mathbb F_{q^2}$, for any prime power $q$, the Suzuki curve $\mathcal S_q$ is maximal over $\mathbb F_{q^4}$, for $q=2^{2h+1}$, $h\geq1$ and the Ree curve $\mathcal R_q$ is maximal over $\mathbb F_{q^6}$, for $q=3^{2h+1}$, $h\geq0$. In this paper we show that $\mathcal S_8$ is not Galois covered by $\mathcal H_{64}$. We also give a proof for an unpublished result due to Rains and Zieve stating that $\mathcal R_3$ is not Galois covered by $\mathcal H_{27}$.
Furthermore, we determine the spectrum of genera of Galois subcovers of $\mathcal H_{27}$, and we point out that some Galois subcovers of $\mathcal R_3$ are not Galois subcovers of $\mathcal H_{27}$.
\end{abstract}

\maketitle

\section{Introduction}

Let $q$ be a prime power, $\mathbb F_{q^2}$ be the finite field with $q^2$ elements, and $\mathcal X$ be an $\mathbb F_{q^2}$-rational curve, i.e. a projective, absolutely irreducible, non-singular algebraic curve defined over $\mathbb F_{q^2}$.
The curve $\mathcal X$ is called $\mathbb F_{q^2}$-maximal if the number $|\mathcal X(\mathbb F_{q^2})|$ of its $\mathbb F_{q^2}$-rational points attains the Hasse-Weil upper bound
$$ q^2+1+2gq, $$
where $g$ is the genus of $\mathcal X$. Maximal curves have interesting properties and have also been investigated for their applications in Coding Theory. Surveys on maximal curves are found in \cite{FT,G,G2,GS,vdG,vdG2} and \cite[Chapt. 10]{HKT}.

By a result commonly attributed to Serre, see \cite[Prop. 6]{L}, any $\mathbb F_{q^2}$-rational curve which is $\fqs$-covered by an $\fqs$-maximal curve is also $\fqs$-maximal. In particular, $\fqs$-maximal curves can be obtained as Galois $\fqs$-subcovers of an $\fqs$-maximal curve $\cX$, that is as quotient curves $\cX/G$ for a finite $\fqs$-automorphism group $G\leq\aut(\cX)$.
Most of the known maximal curves are Galois subcovers of one of the Deligne-Lusztig curves; see e.g. \cite{GSX,CKT2,GHKT2} for subcovers of the Hermitian curve $\cH_q: X^{q+1}+Y^{q+1}+T^{q+1}=0$, \cite{GKT,P} for subcovers of the Suzuki curve $\cS_q: Y^q+Y=X^{q_0}(X^q+X)$, with $q=2q_0^2$, $q_0=2^h$, $h\geq0$, \cite{CO,CO2,Pe} for subcovers of  the Ree curve $\cR_q:Y^q-Y=X^{q_0}(X^q-X), Z^q-Z=X^{q_0}(Y^q-Y)$, with $q=3q_0^2$, $q_0=3^h$, $h\geq0$, and the references therein.

The first example of a maximal curve not Galois covered by the Hermitian curve was discovered by Garcia and Stichtenoth \cite{GS3}. This curve is $\F_{3^6}$-maximal and not Galois covered by $\cH_{27}$. It is a special case of the $\F_{q^6}$-maximal GS curve, which was recently shown not to be Galois covered by $\cH_{q^3}$ for any $q>3$ \cite{GMZ,Mak}. Giulietti and Korchm\'aros \cite{GK} provided an $\F_{q^6}$-maximal curve, nowadays referred to as the GK curve, which is not covered by the Hermitian curve $\cH_{q^3}$ for any $q>2$. In \cite{TTT,GQZ}, some subcovers of the GK curve were shown not to be covered, or Galois covered, by the Hermitian curve. Garcia, G\"uneri, and Stichtenoth \cite{GGS} generalized the GK curve to an $\F_{q^{2n}}$-maximal curve, for any $q$ and any odd $n\geq3$. The generalized GK curve is not Galois covered by $\cH_{q^n}$ for any $n\geq5$, as shown in \cite{DM} for $q>2$ and in \cite{GMZ} for $q=2$.

It is a challenging task to decide whether a DL-curve of Ree or Suzuki type is a Galois subcover of the Hermitian curve.
In this paper we prove the following results.


\begin{theorem}\label{teoremaS8}
The Suzuki curve $\cS_8$ is not Galois covered by the Hermitian curve $\cH_{64}$.
\end{theorem}

\begin{theorem}\label{teoremaR3}
The Ree curve $\cR_3$ is not Galois covered by the Hermitian curve $\cH_{27}$.
\end{theorem}

\begin{proposition} \label{S2}
The Suzuki curve $\cS_2$ is Galois covered by the Hermitian curve $\cH_4$.
\end{proposition}


We note that Theorem \ref{teoremaR3} is an unpublished result due to Rains and Zieve. 

We give an outline of the proofs of Theorems \ref{teoremaS8} and \ref{teoremaR3}. We first bound the possible degrees $d$ of putative Galois coverings $\cH_{64}\rightarrow\cS_8$ and $\cH_{27}\rightarrow\cR_3$ from the Riemann-Hurwitz formula; see \cite[Theorem 3.4.13]{Sti}. Then, for each possible value of $d$, we investigate all subgroups $G$ of $\PGU(3,64)$ and $\PGU(3,27)$ having order $d$.
The structure of $G$ allows us to estimate the contribution to the degree $\Delta$ of the different divisor for each element of $G$; see \cite[Theorem 3.8.7]{Sti}. In most cases, we get that the genus of the quotient curve is different from that of $\cS_8$ and $\cR_3$. Sometimes, a deeper investigation of the automorphism group of the quotient curve $\cH_{64}/G$ or $\cH_{27}/G$ is needed.
As a by-product, we describe in Proposition \ref{H/G} the unique quotient curve $\cX$ of $\cH_{64}$ which has the same genus
 as $\cS_8$; $\cX$ is defined over $\mathbb F_8$. Since $\cX$ is not isomorphic to $\cS_8$, a result by Fuhrmann and Torres \cite[Theorem 5.1]{FT} implies that $\cX$ is not $\mathbb F_8$-optimal.
We also describe in Proposition \ref{H27/G} the quotient curves of $\cH_{27}$ which have the same genus of $\cR_3$.
More generally, we determine in Theorem \ref{spettro} the spectrum of genera of Galois subcovers of $\cH_{27}$. For all $g>1$, we classify all subgroups $G\leq\PGU(3,27)$ such that the quotient curve $\cH_{27}/G$ has genus $g$. We point out that some quotient curves of $\cR_3$, studied by \c{C}ak\c{c}ak and \"Ozbudak \cite{CO}, are not quotient curves of $\cH_{27}$; see Corollary \ref{noncoperte}.

Results by Garcia, Stichtenoth, and Xing \cite{GSX} on the automorphism groups of the Hermitian curve $\cH_q$ fixing an $\mathbb F_{q^2}$-rational point of $\cH_q$ will be used. We also rely on classical classification results of subgroups of $\PSU(3,q)$ by Mitchell \cite{M} and Hartley \cite{H}.

We classify the elements of $\PGU(3,q)$ in terms of their orders and their action on $\PG(2,\overline\F_q)$ and $\cH_q$. In this way, we get the contribution to $\Delta$ of any element of $\PGU(3,q)$, in terms of its geometric properties; see Theorem \ref{caratteri}. This is a result of independent interest, which extends \cite[Lemma 4.1]{DM}.

The paper is organized as follows. In Section \ref{preliminari} we present the preliminary results on quotient curves of the Hermitian curve and the proof of Proposition \ref{S2}. Sections \ref{sezioneS8} and \ref{sezioneR3} contain the proofs of Theorems \ref{teoremaS8} and \ref{teoremaR3}, respectively.
Section \ref{sezioneSpettro} provides the spectrum of genera of quotient curves of $\cH_{27}$ and three examples of quotient curves of $\cR_3$ which are not quotient curves of $\cH_{27}$.

\section{Preliminary results}\label{preliminari}
Throughout this paper, $q=p^n$, where $p$ is a prime number and $n$ is a positive integer. The Deligne-Lusztig curves defined over a finite field $\mathbb{F}_q$ were originally introduced in \cite{DL}. Other than the projective line, there are three families of Deligne-Lusztig curves, named Hermitian curves, Suzuki curves and Ree curves. These curves  are maximal over some finite field containing $\mathbb{F}_q$. The following descriptions with explicit equations come from \cite{Ha,HaP}.
\begin{itemize}
\item The Hermitian curve $\mathcal H_q$ arises from the algebraic group $^2A_2(q)={\rm PGU}(3,q)$ of order $(q^3+1)q^3(q^2-1)$. It has genus $q(q-1)/2$ and is $\mathbb F_{q^2}$-maximal. Two ($\mathbb F_{q^2}$-projectively equivalent) nonsingular plane models of $\mathcal H_q$ are the Fermat curve with homogeneous  equation
\begin{equation}\label{fermat}
X^{q+1}+Y^{q+1}+T^{q+1}=0
\end{equation}
and the norm-trace curve with homogeneous equation
\begin{equation}\label{normatraccia}
Y^{q+1}=X^qT+XT^q .
\end{equation}
The automorphism group $\aut(\cH_q)$ is isomorphic to the projective unitary group $\PGU(3,q)$, and it acts on the set $\cH_q(\mathbb{F}_{q^2})$ of all $\mathbb{F}_{q^2}$-rational points of $\cH_q$ as $\PGU(3,q)$ in its usual $2$-transitive permutation representation.
\item The Suzuki curve $\mathcal S_q$ arises from the simple Suzuki group $^2B_2(q)=Sz(q)$. It has genus $q_0(q-1)$ and is $\mathbb F_{q^4}$-maximal, where $q=2q_0^2$, $q_0=2^h$, $h\geq1$. A (singular) plane model of $\mathcal S_q$ is given by the affine equation $$ Y^q+Y=X^{q_0}(X^q+X) .$$
    The automorphism group $\aut(\mathcal S_q)$ is isomorphic to a subgroup of the projective group $\PGL(4,q)$ preserving the Suzuki-Tits ovoid $O_{S}$ in $\PG(3,q)$, and it acts on $O_{S}$ as $Sz(q)$ in its usual $2$-transitive permutation representation.
\item The Ree curve $\mathcal R_q$ arises from the simple Ree group $^2G_2(q)=Ree(q)$. It has genus $3q_0(q-1)(q+q_0+1)/2$ and is $\mathbb F_{q^6}$-maximal, where $q=3q_0^2$, $q_0=3^h$, $h\geq0$. A (singular) space model of $\mathcal R_q$ is given by the affine equations $$ Y^q-Y=X^{q_0}(X^q-X),\quad Z^q-Z=X^{q_0}(Y^q-Y).$$ The automorphism group $\aut(\mathcal R_q)$ is isomorphic to a subgroup of the projective group $\PGL(7,q)$ preserving the Ree-Tits ovoid $O_{R}$ in $\PG(6,q)$, and it acts on $O_{R}$ as $Ree(q)$ in its usual $2$-transitive permutation representation.
\end{itemize}

We extend the definition of a Suzuki curve to the case $q=2$. A (singular) plane model of $\cS_2$ is given by the
$$ \cS_2:\quad Y^2+Y=X(X^2+X). $$
In particular, $\cS_2$ is an elliptic and $\F_{2^4}$-maximal curve.

The combinatorial properties of $\cH_q(\mathbb{F}_{q^2})$ can be found in \cite{HP}. The size of $\cH_q(\mathbb{F}_{q^2})$ is equal to $q^3+1$, and a line of $PG(2,q^2)$ has either $1$ or $q+1$ common points with $\cH_q(\mathbb{F}_{q^2})$, that is, it is either a $1$-secant or a chord of $\cH_q(\mathbb{F}_{q^2})$. Furthermore, a unitary polarity is associated with $\cH_q(\mathbb{F}_{q^2})$ whose isotropic points are those of $\cH_q(\mathbb{F}_{q^2})$ and isotropic lines are the $1$-secants of $\cH_q(\mathbb{F}_{q^2})$, that is, the tangents to $\cH_q$ at the points of $\cH_q(\mathbb{F}_{q^2})$.

From Group theory we need the classification of all maximal subgroups of the projective special subgroup $\PSU(3,q)$ of $\PGU(3,q)$, going back to Mitchell and Hartley; see \cite{M}, \cite{H}, \cite{HO}.
\begin{theorem} \label{Mit} Let $d={\rm gcd}(3,q+1)$. Up to conjugacy, the subgroups below give a complete list of maximal subgroups of $\PSU(3,q)$.

\begin{itemize}
\item[(i)] the stabilizer of an $\F_{q^2}$-rational point of $\cH_q$. It has order $q^3(q^2-1)/d$;
\item[(ii)] the stabilizer of an $\F_{q^2}$-rational point off $\cH_q$ $($equivalently the stabilizer of a chord of $\cH_q(\mathbb{F}_{q^2}))$. It has order $q(q-1)(q+1)^2/d$;
\item[(iii)] the stabilizer of a self-polar triangle with respect to the unitary polarity associated to $\cH_q(\mathbb{F}_{q^2})$. It has order $6(q+1)^2/d$;
\item[(iv)] the normalizer of a (cyclic) Singer subgroup. It has order $3(q^2-q+1)/d$ and preserves a triangle
in $\PG(2,q^6)\setminus\PG(2,q^2)$ left invariant by the Frobenius collineation $\Phi_{q^2}:(X,Y,T)\mapsto (X^{q^2},Y^{q^2},T^{q^2})$ of $\PG(2,\bar{\mathbb{F}}_{q})$;

{\rm for $p>2$:}
\item[(v)] ${\rm PGL}(2,q)$ preserving a conic;
\item[(vi)] $\PSU(3,p^m)$ with $m\mid n$ and $n/m$ odd;
\item[(vii)] subgroups containing $\PSU(3,p^m)$ as a normal subgroup of index $3$, when $m\mid n$, $n/m$ is odd, and $3$ divides both $n/m$ and $q+1$;
\item[(viii)] the Hessian groups of order $216$ when $9\mid(q+1)$, and of order $72$ and $36$ when $3\mid(q+1)$;
\item[(ix)] ${\rm PSL(2,7)}$ when $p=7$ or $-7$ is not a square in $\mathbb{F}_q$;
\item[(x)] the alternating group $Alt(6)$ when either $p=3$ and $n$ is even, or $5$ is a square in $\mathbb{F}_q$ but $\mathbb{F}_q$ contains no cube root of unity;
\item[(xi)] the symmetric group $Sym(6)$ when $p=5$ and $n$ is odd;
\item[(xii)] the alternating group $Alt(7)$ when $p=5$ and $n$ is odd;

{\rm for $p=2$:}
\item[(xiii)] $\PSU(3,2^m)$ with $m\mid n$ and $n/m$ an odd prime;
\item[(xiv)] subgroups containing $\PSU(3,2^m)$ as a normal subgroup of index $3$, when $n=3m$ with $m$ odd;
\item[(xv)] a group of order $36$ when $n=1$.
\end{itemize}
\end{theorem}
In our investigation it is useful to know how an element of $\PGU(3,q)$ of a given order acts on $\PG(2,\bar{\mathbb{F}}_{q})$, and in particular on $\cH_q(\mathbb{F}_{q^2})$. This can be obtained as a corollary of Theorem \ref{Mit}, and is stated in Lemma $2.2$ with the usual terminology of collineations of projective planes; see \cite{HP}. In particular, a linear collineation $\sigma$ of $\PG(2,\bar{\mathbb{F}}_q)$ is a $(P,\ell)$-\emph{perspectivity}, if $\sigma$ preserves  each line through the point $P$ (the \emph{center} of $\sigma$), and fixes each point on the line $\ell$ (the \emph{axis} of $\sigma$). A $(P,\ell)$-perspectivity is either an \emph{elation} or a \emph{homology} according as $P\in \ell$ or $P\notin\ell$. A $(P,\ell)$-perspectivity is in  $\PGL(3,q^2)$ if and only if its center and its axis are in $\PG(2,\mathbb{F}_{q^2})$.
\begin{lemma}\label{classificazione}
For a nontrivial element $\sigma\in\PGU(3,q)$, one of the following cases holds.
\begin{itemize}
\item[(A)] ${\rm ord}(\sigma)\mid(q+1)$. Moreover, $\sigma$ is a homology whose center $P$ is a point off $\cH_q$ and whose axis $\ell$ is a chord of $\cH_q(\mathbb{F}_{q^2})$ such that $(P,\ell)$ is a pole-polar pair with respect to the unitary polarity associated to $\cH_q(\mathbb{F}_{q^2})$.
\item[(B)] ${\rm ord}(\sigma)$ is coprime with $p$. Moreover, $\sigma$ fixes the vertices $P_1,P_2,P_3$ of a non-degenerate triangle $T$.
\begin{itemize}
\item[(B1)] The points $P_1,P_2,P_3$ are $\fqs$-rational, $P_1,P_2,P_3\notin\cH_q$ and the triangle $T$ is self-polar with respect to the unitary polarity associated to $\cH_q(\mathbb{F}_{q^2})$. Also, $\ord(\sigma)\mid(q+1)$.
\item[(B2)] The points $P_1,P_2,P_3$ are $\fqs$-rational, $P_1\notin\cH_q$, $P_2,P_3\in\cH_q$. 
     Also, $\ord(\sigma)\mid(q^2-1)$ and $\ord(\sigma)\nmid(q+1)$.
\item[(B3)] The points $P_1,P_2,P_3$ have coordinates in $\F_{q^6}\setminus\F_{q^2}$, $P_1,P_2,P_3\in\cH_q$. 
    Also, $\ord(\sigma)\mid (q^2-q+1)$.
\end{itemize}
\item[(C)] ${\rm ord}(\sigma)=p$. Moreover, $\sigma$ is an elation whose center $P$ is a point of $\cH_q$ and whose axis $\ell$ is a tangent of $\cH_q(\mathbb{F}_{q^2})$ such that $(P,\ell)$ is a pole-polar pair with respect to the unitary polarity associated to $\cH_q(\mathbb{F}_{q^2})$.
\item[(D)] ${\rm ord}(\sigma)=p$ with $p\ne2$, or ${\rm ord}(\sigma)=4$ and $p=2$. Moreover, $\sigma$ fixes an $\fqs$-rational point $P$, with $P \in \cH_q$, and a line $\ell$ which is a tangent  of $\cH_q(\mathbb{F}_{q^2})$, such that $(P,\ell)$ is a pole-polar pair with respect to the unitary polarity associated to $\cH_q(\mathbb{F}_{q^2})$.
\item[(E)] $p\mid{\rm ord}(\sigma)$, $p^2\nmid{\rm ord}(\sigma)$, and ${\rm ord}(\sigma)\ne p$. Moreover, $\sigma$ fixes two $\fqs$-rational points $P,Q$, 
     with $P\in\cH_q$, $Q\notin\cH_q$. 
\end{itemize}
\end{lemma}

\begin{proof}
Let $p\mid{\rm ord}(\sigma)$, ${\rm ord}(\sigma)\ne p$, and $(p,\ord(\sigma))\ne(2,4)$. By \cite[\S 2 p. 212]{M} and \cite[pp. 141-142]{H}, the fixed elements of $\sigma$ are two points $P,Q$, the line $PQ$, and another line $\ell$ through $P$. Also, $p^2\nmid\ord(\sigma)$. The Frobenius collineation $\Phi_{q^2}:(X,Y,T)\mapsto(X^{q^2},Y^{q^2},T^{q^2})$ commutes with $\sigma$. Hence $\Phi_{q^2}$ acts on $\{P,Q\}$, and $P,Q$ are $\F_{q^4}$-rational. If $R\in\{P,Q\}$ is the pole of $PQ$, then $R\in\cH_q$. Since $\cH_q$ has no points with coordinates in $\F_{q^4}\setminus\F_{q^2}$, $R$ is $\fqs$-rational. Thus the line $PQ$ is a tangent of $\cH_q(\mathbb{F}_{q^2})$ at $R$. Hence the pole of $\ell$ is $\fqs$-rational and off $\ell$. Therefore $R=P$ and the assertions of Case (E) follow.

Let $\ord(\sigma)=p$, and let $\cH_q$ have equation \eqref{normatraccia}. Up to conjugation, $\sigma$ is contained in the Sylow $p$-subgroup $S$ of $\PGU(3,q)$ defined by $S=\left\{ \tau_{1,b,c} \mid b,c\in\fqs,\;b^{q+1}=c^q+c \right\}$, where
\begin{equation}\label{tau}
\tau_{1,b,c}=\begin{pmatrix} 1 & b^q & c \\ 0 & 1 & b \\ 0 & 0 & 1\end{pmatrix}.
\end{equation}
Hence $\sigma$ fixes the $\fqs$-rational point $P_\infty=(1,0,0)\in\cH_q$ and its polar line $\ell_\infty:T=0$, which satisfies $\ell_\infty\cap\cH_q=\{P_\infty\}$. If $p=2$, then $\sigma$ is of type $\tau_{1,0,c}$, and $\sigma$ is an elation with center $P_\infty$ and axis $\ell_\infty$, which is Case (C). If $p\ne2$, then by \cite[\S 2 p. 212]{M} $\sigma=\tau_{1,b,c}$ satisfies either Case (C) or Case (D). By direct computation, Cases (C) and (D) correspond to $b=0$ and $b\ne0$, respectively.

Let $p\nmid\ord(\sigma)$. By \cite[\S 2 p. 212]{M} and \cite[pp. 141-142]{H}, either $\sigma$ fixes a point $P$ and a line $\ell$ pointwise, or $\sigma$ fixes exactly three non-collinear points.

Assume that the former case holds. Then $P$ and $\ell$ are fixed by $\Phi_{q^2}$. Hence, they are defined over $\fqs$. We have $P\notin\cH_q$. In fact, if $P\in\cH_q$, then the tangent to $\cH_q$ at $P$ intersect $\ell$ at an $\fqs$-rational point $Q\notin\cH_q$, and the $\fqs$-rational pole $R$ of $\ell$ lies on $\ell$, hence also on $\cH_q$. For any $\fqs$-rational point $\bar P$ of $\ell\setminus\{R\}$, we have that $\bar P\notin\cH_q$ and the polar line of $\bar P$ intersects $\ell$ at another $\fqs$-rational point of $\ell$. Since the line $PQ$ is the polar line of $P$, this is a contradiction. Therefore, $\ell$ is the polar line of $P$, and $\ell$ is a chord of $\cH_q(\mathbb{F}_{q^2})$. Now we show that $\ord(\sigma)\mid(q+1)$. Let $\cH_q$ have equation \eqref{fermat}. Up to conjugation, $P=(0,0,1)$ and $\ell:T=0$. Hence $\sigma$ is a diagonal matrix of the form ${\rm diag}(\lambda,1,1)$, which implies $\ord(\sigma)=\ord(\lambda)$ with $\ord(\lambda)\mid(q+1)$. This shows that $\sigma$ satisfies Case (A).

Now assume that $\sigma$ fixes exactly the vertices $P_1,P_2,P_3$ of a triangle $T$.
\begin{itemize}
\item Suppose that $P_1$, $P_2$, and $P_3$ are $\fqs$-rational.
If $P_1,P_2,P_3\notin\cH_q$, then $P_jP_k$ is the polar line of $P_i$, for $\{i,j,k\}=\{1,2,3\}$. Let $\cH_q$ have equation \eqref{fermat}. Up to conjugation $P_1$, $P_2$, and $P_3$ are the fundamental points. Thus $\sigma$ is a diagonal matrix and $\ord(\sigma)\mid(q+1)$, which is Case (B1).
Assume $P_2\in\cH_q$. Then the polar line $\ell_2$ of $P_2$ is either $P_1P_2$ or $P_2P_3$, say $P_1P_2$. The polar line $\ell_3$ of $P_3$ is either $P_1P_3$ or $P_2P_3$, whence $P_3\in\ell_3$ and $P_3\in\cH_q$. Then $\ell_3\cap\cH_q=\{P_3\}$, and hence $\ell_3$ is $P_1P_3$. This implies that $P_2P_3$ is the polar line of $P_1$ and $P_1\notin\cH_q$. Let $\cH_q$ have equation \eqref{normatraccia}. Up to conjugation, $P_2=(1,0,0)$ and $P_3=(0,0,1)$. Thus $P_1=(0,1,0)$ and $\sigma$ is the diagonal matrix ${\rm diag}(\mu^{q+1},\mu,1)$ for some $\mu\in\F_{q^2}^*$. Since $\sigma$ is not a homology, $\ord(\sigma)=\ord(\mu)$ does not divide $q+1$. This is Case (B2).
\item Suppose that $P_1$ has coordinates in $\F_{q^6}\setminus\fqs$. The orbit of $P_1$ under $\Phi_{q^2}$ is $\{P_1,P_2,P_3\}$. Hence, $P_2$ and $P_3$ have coordinates in $\F_{q^6}\setminus\fqs$ as well. Assume $P_1\in\cH_q$. Then the polar line $\ell_1$ of $P_1$ is tangent to $\cH_q$ at $P_1$ and $\ell_1$ has exactly another point $P$ in common with $\cH_q$, which is then fixed by $\sigma$. Up to reordering, $P=P_2$. In the same way, $P_3\in\cH_q$ and the polar line of $P_1,P_2,P_3$ are $P_1P_2$, $P_2P_3$, $P_3P_1$, respectively. Let $H\leq\PGU(3,q)$ be the Singer group consisting of the elements of $\PGU(3,q)$ fixing the triangle $T$. Then $H$ has order $q^2-q+1$; see \cite{M} and \cite{H}. Since $\sigma\in H$, $\ord(\sigma)\mid(q^2-q+1)$ and $\sigma$ satisfies Case (B3).

Elements satisfying Case (B3) do exist; see for instance \cite[Lemma 4.4]{CKT}. The number $k$ of triangles $T$ whose vertices $Q_1, Q_2, Q_3$ are such that $Q_i\in\PG(2,q^6)\setminus\PG(2,q^2)$ and there exists some $\sigma\in\PGU(3,q)$ stabilizing $T$, is equal to the index in $\PGU(3,q)$ of the normalizer $N$ of $H$. By \cite{M} and \cite{H} (see Case (iv) in Theorem \ref{Mit}), $|N|=3(q^2-q+1)$. Hence $k=q^3(q+1)^2(q-1)/3$. By direct computation, $k$ is equal to the number of triangles $T'$ whose vertices $Q_1', Q_2', Q_3'$ are such that $Q_i'\in\PG(2,q^6)\setminus\PG(2,q^2)$ and $Q_i'\in\cH_q$, $i=1,2,3$. Therefore, it is not possible that $P_1,P_2,P_3$ have coordinates in $\F_{q^6}\setminus\fqs$ and $P_1\notin\cH_q$.
\item The case that $P_1$ has coordinates in $\F_{q^4}\setminus\fqs$ cannot occur. In fact, since $\Phi_{q^2}$ acts on $\{P_1,P_2,P_3\}$, if $P_1\in\PG(2,\F_{q^4})\setminus\PG(2,\F_{q^2})$, then up to reordering $P_2\in\PG(2,\F_{q^4})\setminus\PG(2,\F_{q^2})$ and $P_3\in\PG(2,q^2)$. Let $i,j\in\{1,2,3\}$, $i\ne j$. By \cite[\S 2 p. 212]{M} and \cite[pp. 141-142]{H}, any power of $\sigma$ either fixes the line $P_iP_j$ pointwise or has no fixed points on $P_iP_j\setminus\{P_i,P_j\}$. Thus $\sigma$ has long orbits on $P_iP_j\setminus\{P_i,P_j\}$. In particular, $\ord(\sigma)$ divides the number of $\fqs$-rational points of both $P_1P_2$ and $P_1P_3\setminus\{P_3\}$, a contradiction.
\end{itemize}
\end{proof}
\vspace*{.2cm}

Throughout the paper, a nontrivial element of $\PGU(3,q)$ is said to be of type (A), (B), (B1), (B2), (B3), (C), (D), or (E), as given in Lemma \ref{classificazione}. Moreover, $G$ always stands for a subgroup of $\PGU(3,q)$.

\begin{lemma}\label{azione}
Let $H$ be a normal subgroup of $G$. Let $A$ be the set of points of $\PG(2,\bar\F_q)$ fixed by every element of $H$, and $B=A\cap\cH_q$. Then $G$ acts on $B$ and on $A\setminus B$.
\end{lemma}


\begin{lemma}\label{pgruppo1}
Let $H$ be a $m$-subgroup of $\PGU(3,q)$, where $m\notin\{2,3\}$ is a prime divisor of $q+1$. Then $H$ is abelian. Also, the nontrivial elements of $H$ are either of types {\rm(A)} or {\rm(B1)}, and in the latter case the fixed triangle $T$ is the same for every element of $H$. In addiction, if $H$ is a Sylow $m$-subgroup of $\PGU(3,q)$, then the unique fixed points of $H$ are the vertices of $T$ and $H$ is the direct product of two cyclic groups whose nontrivial elements are of type {\rm(A)}.
\end{lemma}

\begin{proof}
Since $p\notin\{2,3\}$, the maximum power of $m$ dividing $|\PGU(3,q)|$ is a square, say $m^{2s}$.
Let $\cH_q$ have equation \eqref{fermat}, and define
\begin{equation}\label{sylow1}
K=\{\diag(\lambda,\mu,1)\mid \lambda^s=\mu^s=1\}\cong\{\diag(\lambda,1,1)\mid \lambda^s=1\}\times\{\diag(1,\mu,1)\mid \mu^s=1\}.
\end{equation}
Then $K$ is an abelian Sylow $m$-subgroup of $\PGU(3,q)$, whose fixed points are the fundamental points. Also, the nontrivial elements of $K$ are either of type {\rm(A)} or {\rm(B1)}. Up to conjugation, $H$ is contained in $K$ and the claim follows.
\end{proof}

\begin{lemma}\label{pgruppo2}
Let $H$ be a $m$-subgroup of $\PGU(3,q)$, where $m$ is an odd prime divisor of $q-1$. Then $H$ is abelian and the unique fixed points of $H$ are the vertices of a triangle $T$.
\end{lemma}

\begin{proof}
Let $\cH_q$ have equation \eqref{normatraccia}, and define
\begin{equation}\label{sylow2}
K=\{\diag(a^{q+1},a,1)\mid a\in\mathbb F_{q^2}^*\}.
\end{equation}
Then $K$ is an abelian Sylow $m$-subgroup of $\PGU(3,q)$, and the nontrivial elements of $K$ fix exactly the fundamental points. Up to conjugation, $H$ is contained in $K$ and the claim follows.
\end{proof}

%

\begin{lemma}\label{stabilizzatore}
Let $p\in\{2,3\}$. If $G$ has a nontrivial normal subgroup $H$ of prime order other than $p$, then $p^2\nmid|G|$.
\end{lemma}

\begin{proof}
Assume by contradiction that $p^2\mid|G|$ and let $\sigma\in H$. By Lemma \ref{classificazione}, the type of $\sigma$ is either (A) or (B). Suppose that $\sigma$ is of type (A). Then, since $H=\langle \sigma\rangle$, all nontrivial elements of $H$ are of type (A) and they have the same center $P$ and axis $\ell$. On the other hand, by Lemma \ref{azione}, any $p$-element of $G$ fixes $P$ and acts on $\ell$; a contradiction by Lemma \ref{classificazione}. Suppose that $\sigma$ is of type (B). Then, since $H=\langle \sigma\rangle$, all nontrivial elements of $H$ are of type (B) and they fix the same triangle $T$. By Lemma \ref{azione}, $G$ preserves $T$. Hence, by the orbit-stabilizer theorem, the elements of $G$ fixing $T$ pointwise form a subgroup $M$ of index $1$, $2$, or $3$. In all cases, $M$ contains a $p$-element of type (A) or type (B), a contradiction by Lemma \ref{classificazione}.
\end{proof}
\vspace*{.2cm}

From Function field theory we need the Riemann-Hurwitz formula; see \cite[Theorem 3.4.13]{Sti}. Every subgroup $G$ of $\PGU(3,q)$ produces a quotient curve $\cH_q/G$, and the cover $\cH_q\rightarrow\cH_q/G$ is a Galois cover defined over $\mathbb{F}_{q^2}$  where the degree of the different divisor $\Delta$ is given by the Riemann-Hurwitz formula, namely
$\Delta=(2g(\cH_q)-2)-|G|(2g(\cH_q/G)-2)$. On the other hand, $\Delta=\sum_{\sigma\in G\setminus\{id\}}i(\sigma)$, where $i(\sigma)\geq0$ is given by the Hilbert's different formula \cite[Thm. 3.8.7]{Sti}, namely
\begin{equation}\label{contributo}
\textstyle{i(\sigma)=\sum_{P\in\cH_q(\bar\F_q)}v_P(\sigma(t)-t),}
\end{equation}
where $t$ is a local parameter at $P$.

By analyzing the geometric properties of the elements $\sigma \in \PGU(3,q)$, it turns out that there are only a few possibilities for $i(\sigma)$.
This is obtained as a corollary of Lemma \ref{classificazione} and stated in the following proposition.

\begin{theorem}\label{caratteri}
For a nontrivial element $\sigma\in \PGU(3,q)$ one of the following cases occurs.
\begin{enumerate}
\item If $\ord(\sigma)=2$ and $2\mid(q+1)$, then $\sigma$ is of type {\rm(A)} and $i(\sigma)=q+1$.
\item If $\ord(\sigma)=3$, $3 \mid(q+1)$ and $\sigma$ is of type {\rm(B3)}, then $i(\sigma)=3$.
\item If $\ord(\sigma)\ne 2$, $\ord(\sigma)\mid(q+1)$ and $\sigma$ is of type {\rm(A)}, then $i(\sigma)=q+1$.
\item If $\ord(\sigma)\ne 2$, $\ord(\sigma)\mid(q+1)$ and $\sigma$ is of type {\rm(B1)}, then $i(\sigma)=0$.
\item If $\ord(\sigma)\mid(q^2-1)$ and $\ord(\sigma)\nmid(q+1)$, then $\sigma$ is of type {\rm(B2)} and $i(\sigma)=2$.
\item If $\ord(\sigma)\ne3$ and $\ord(\sigma)\mid(q^2-q+1)$, then $\sigma$ is of type {\rm(B3)} and $i(\sigma)=3$.
\item If $p=2$ and $\ord(\sigma)=4$, then $\sigma$ is of type {\rm(D)} and $i(\sigma)=2$.
\item If $\ord(\sigma)=p$, $p \ne2$ and $\sigma$ is of type {\rm(D)}, then $i(\sigma)=2$.
\item If $\ord(\sigma)=p$ and $\sigma$ is of type {\rm(C)}, then $i(\sigma)=q+2$.
\item If $\ord(\sigma)\ne p$, $p\mid\ord(\sigma)$ and $\ord(\sigma)\ne4$, then $\sigma$ is of type {\rm(E)} and $i(\sigma)=1$.
\end{enumerate}
\end{theorem}

\begin{proof}
Suppose $p\nmid\ord(\sigma)$. Then by \cite[Theorem 11.74]{HKT} $i(\sigma)$ equals the number of points of $\cH_q$ fixed by $\sigma$.
Also, for $q$ odd all involutions are conjugated and are of type (A), by \cite[Lemma 2.2 (ii)]{KOS}.
Therefore Cases (1) - (6) follow from Lemma \ref{classificazione}.

Suppose $\ord(\sigma)=p$, or $p=2$ and $\ord(\sigma)=4$.
As in the proof of Lemma \ref{classificazione}, we can assume that $\sigma$ has the form $\tau_{1,b,c}$ defined in \eqref{tau}. By direct computation, $\sigma$ is of type (C) or (D) if and only if $b=0$ or $b\ne0$, respectively.
By \cite[Eq. (2.12)]{GSX}, $b=0$ or $b\ne0$ if and only if $i(\sigma)=q+2$ or $i(\sigma)=2$, respectively. From this, Cases (8) and (9) follow.
Since $(p,\ord(\tau_{1,b,c}))=(2,4)$ implies $b\ne0$, Case (7) follows as well.

Suppose $p\mid\ord(\sigma)$, $\ord(\sigma)\ne p$, and $\ord(\sigma)\ne4$. By \cite[\S 2 p. 212]{M} and \cite[pp. 141-142]{H}, $\sigma$ is of type (E). Let $P\in\cH_q$ be the unique fixed point of $\sigma$ on $\cH_q$. By \cite[Theorem 11.74]{HKT}, $\sigma$ is in the stabilizer of $P$ but is not a $p$-element. Hence $i(\sigma)=1$.
Since Cases (A) - (E) in Lemma \ref{classificazione} cover all nontrivial elements of $\PGU(3,q)$, Cases (1) - (10) give a complete classification.
\end{proof}
\vspace*{.2cm}

Theorem \ref{caratteri} extends \cite[Lemma 4.1]{DM}, where the result is for $\sigma$ fixing an $\fqs$-rational point of $\cH_q$.

Groups fixing an $\fqs$-rational point of $\cH_q$ are investigated in \cite{GSX}.

\begin{theorem}\label{fissa2}
{\rm \cite[Thm. 3.3 and Eq. (2.12)]{GSX}}
Let $p=2$. For a positive integer $g$, the following assertions are equivalent.
\begin{enumerate}
\item There exists a $2$-subgroup $G\leq\PGU(3,q)$ such that $g=g(\cH_q/G)$.
\item $g=2^{n-v-1}(2^{n-w}-1)$ with $0\le v\le n-1$ and $0\le w\le n-1$, and there exist additive subgroups $V\subseteq\fqs$ and $W\subseteq\F_q$ of order $\ord(V)=2^v$ and $\ord(W)=2^w$, such that $V^{q+1}=\{b^{q+1}\mid b\in V\}$ is contained in $W$.
\end{enumerate}
\end{theorem}
Assume that assertions {\rm (1)} and {\rm (2)} hold, and let $\cH_q$ have equation \eqref{normatraccia}. Up to conjugation the unique point of $\cH_q$ fixed by every element of $G$ is $P_\infty=(1,0,0)$, and the elements of $G$ have the form \eqref{tau}.
Then $|G|=2^{v+w}$ and the additive subgroups $\{b\in\fqs\mid\tau_{1,b,c}\in G\}\leq\fqs$ and $\{c\in\fqs\mid\tau_{1,0,c}\in G\}\leq\fq$ have order $2^v$ and $2^w$, respectively. In particular, the number of involutions of $G$ equals $2^w-1$.

\begin{theorem}\label{fissa}
{\rm \cite[Thm. 4.4 and Eq. (2.12)]{GSX}}
Let $G$ fix an $\fqs$-rational point $P\in\cH_q$, and let $|G|=m\cdot p^u$ with $m>1$, $m$ coprime with $p$. Then $\cH_q/G$ has genus
$$ g(\cH_q/G)=\frac{\left(q-p^w\right)\left(q-(\gcd(m,q+1)-1)p^v\right)}{2mp^u}\,, $$
where $v,w$ are non-negative integers such that $v+w=u$.
\end{theorem}

Assume that $G$ satisfies the hypotheses of Theorem \ref{fissa} and let $\cH_q$ have equation \eqref{normatraccia}. Up to conjugation $P=(1,0,0)$ and the elements of $G$ have the form
$$ \tau_{a,b,c}=\begin{pmatrix} a^{q+1} & b^q & c \\ 0 & a & b \\ 0 & 0 & 1\end{pmatrix}, $$
with $a,b,c\in\fqs$, $a\ne0$, $b^{q+1}=c^q+c$.
Then the additive subgroups $\{b\in\fqs\mid\tau_{1,b,c}\in G\}$ and $\{c\in\fqs\mid\tau_{1,0,c}\in G\}$ of $\fqs$ have order $p^v$ and $p^w$, respectively.
 In particular, the number of nontrivial elements $\sigma\in G$ with $i(\sigma)=q+2$ equals $p^w-1$.

As a consequence of Theorem \ref{caratteri}, the following result is obtained.

\begin{proposition}
The Suzuki curve $\cS_2$ is a Galois subcover of the Hermitian curve $\cH_4$.
\end{proposition}

\begin{proof}
The Suzuki curve $\cS_2$ has genus $1$ and is $\F_{16}$-maximal.
Let $G\leq\PGU(3,4)$ be a cyclic group of order $4$. By Theorem \ref{caratteri}, the $\F_{16}$-maximal quotient curve $\cH_4/G$ is elliptic.
By \cite[Thm. 77]{Hurt}, there is only one $\F_{16}$-isomorphism class of $\F_{16}$-maximal elliptic curves. Then $\cS_2$ is $\F_{16}$-isomorphic to $\cH_4/G$.
\end{proof}
\vspace*{.2cm}

Throughout the rest of the paper, $C_r$ stands for a cyclic group of order $r$, $S_m$ is a Sylow $m$-subgroup of $G$, and $n_m$ is the number of Sylow $m$-subgroups of $G$.

%

\section{Proof of Theorem \ref{teoremaS8}}\label{sezioneS8}

By absurd, let $G\leq\PGU(3,64)$ be such that $\cS_8\cong\cH_{64}/G$. The order of $\PGU(3,64)$ is equal to $2^{18}\cdot3^2\cdot5^2\cdot7\cdot13^2\cdot37\cdot109$. From the Riemann-Hurwitz formula,
$$ 44<\frac{|\cH_{64}(\F_{8^4})|}{|\cS_8(\F_{8^4})|}\leq|G|\leq\frac{2g(\cH_{64})-2}{2g(\cS_8)-2}\leq 155. $$
Since $|G|$ divides $|\PGU(3,64)|$,
$$|G| \in \{45,48,50,52,56,60,63,64,65,70,72,74,75,78,80,84,90,91,96,$$
$$ \hspace{1.8 cm} 100,104,105,109,111,112,117,120,126,128,130,140,144,148,150\}.$$
The different divisor has degree
\begin{equation}\label{diffS8}
\Delta=(2g(\cH_{64})-2)-|G|(2g(\cS_8)-2) = 4030-26\cdot|G|.
\end{equation}

{\bf Case $|G|=45$.} By Sylow's Third Theorem \cite[Thm. 6.10]{Ro} and Schur-Zassenhaus Theorem \cite[Thm. 9.19]{Ro}, $G$ is the direct product $G=S_3\times C_5$. Then $G$ has $4$ elements of order $5$ and $40$ elements of odd order multiple of $3$. By Theorem \ref{caratteri}, $\Delta\leq 4\cdot65 + 40\cdot2$, contradicting \eqref{diffS8}.

{\bf Case $|G|=48$.} Any group of order $48$ has a normal subgroup of order $8$ or $16$ (see \cite[p. 154 Ex. 10]{Ma}); hence $G$ has a normal $2$-subgroup $N$. By \cite[Theorem 11.74]{HKT}, $N$ has a unique fixed point $P$ on $\cH_{64}$, which is $\F_{64^2}$-rational. By Lemma \ref{azione}, $G$ fixes $P$. From Theorem \ref{fissa},
$$ 14 = \frac{(64-2^w)(64-(\gcd(3,65)-1)2^v)}{2\cdot48}\,, $$
with $v+w=4$. By direct computation, this is not possible.

{\bf Case $|G|=50$.} By Sylow's Third Theorem and Schur-Zassenhaus Theorem, $G$ is a semidirect product $G=S_{5}\rtimes C_2$.
By Theorem \ref{caratteri}, $\Delta=i\cdot65+(24-i)\cdot0+n_2\cdot66+(25-n_2)\cdot1$ with $0\leq i\leq24$ and $n_2\in\{1,5,25\}$. This contradicts \eqref{diffS8}.

{\bf Case $|G|=52$.} By Sylow's Third Theorem, $n_{13}=1$. This contradicts Lemma \ref{stabilizzatore}.

{\bf Case $|G|=56$.} By Sylow's Third Theorem, $n_2=1$ or $n_7=1$. Suppose that $n_2=1$, so that $G=S_2\rtimes C_7$. Then $S_2$ fixes an $\F_{64^2}$-rational point $P\in\cH_{64}$, and $G$ fixes $P$ by Lemma \ref{azione}. By Theorem \ref{fissa},
$$ 14=\frac{(64-2^w)(64-(\gcd(7,65)-1)2^v)}{2\cdot56}\,, $$
with $v+w=2$; this is a contradiction. The case $n_7=1$ is impossible by Lemma \ref{stabilizzatore}.

{\bf Case $|G|=60$.} By \cite[Problem 6.16]{Ro}, either $n_5=1$ or $G$ is isomorphic to the alternating group $Alt(5)$. The case $n_5=1$ is impossible by Lemma \ref{stabilizzatore}; hence $G\cong Alt(5)$. By Theorem \ref{caratteri}, $\Delta=15\cdot66+20\cdot2+i\cdot65+(24-i)\cdot0$, with $0\leq i\leq24$. This contradicts \eqref{diffS8}.

{\bf Case $|G|=63$.} By Theorem \ref{caratteri}, $\Delta=62\cdot2$, contradicting \eqref{diffS8}.

{\bf Case $|G|=64$.} By Theorem \ref{fissa2}, $14 = 2^{6-v-1}(2^{6-w}-1)$ with $0\leq v,w \leq5$. Hence, $v=4$ and $w=3$. Then, by Theorem \ref{fissa2} and Lemma \ref{classificazione}, $G$ has $7$ elements of type (C) and $56$ elements of type (D). By Theorem \ref{caratteri}, $\Delta=7\cdot66+56\cdot2$. This contradicts \eqref{diffS8}.

{\bf Case $|G|=65$.} 
By Lemma \ref{classificazione}, any nontrivial element $\sigma\in G$ is either of type (A) or of type (B1).
If a generator of the cyclic group $G$ is of type (A), then any element is of type (A) and $\Delta=64\cdot65$ by Theorem \ref{caratteri}, contradicting \eqref{diffS8}. If the $48$ generators of $G$ are of type (B1), then $\Delta\leq16\cdot65$ by Theorem \ref{caratteri}. This contradicts \eqref{diffS8}.

{\bf Case $|G|=70$.} By Sylow's Third Theorem, $n_5=n_7=1$ and $n_2\in\{1,5,7,35\}$; hence, $G=C_{35}\rtimes C_2$.
By Theorem \ref{caratteri}, $\Delta=n_2\cdot66+(35-n_2)\cdot1+30\cdot2+i\cdot65+(4-i)\cdot0$ with $0\leq i\leq4$. This contradicts \eqref{diffS8}.

{\bf Case $|G|=72$.} By \cite[Theorem 1]{Mon}, $G$ has a characteristic $3$-subgroup $N$. By Lemma \ref{pgruppo2}, the elements of $N$ are of type (B2) with a common fixed triangle $T$. By Lemma \ref{azione}, $G$ acts on $T$. By the orbit-stabilizer theorem, $G$ contains a $2$-element fixing $T$ pointwise, contradicting Lemma \ref{classificazione}.

{\bf Case $|G|=74$.} For any prime power $q$, $\PSU(3,q)$ has index $\gcd(3,q+1)$ in $\PGU(3,q)$. This implies that, for any maximal subgroup $M\ne\PSU(3,q)$ of $\PGU(3,q)$, $|M|$ divides three times the order of a maximal subgroup of $\PSU(3,q)$. By Theorem \ref{Mit}, $74$ does not divide three times the order of any maximal subgroup of $\PSU(3,64)$, a contradiction.

{\bf Case $|G|=75$.} By Sylow and Schur-Zassenhaus theorems, $G$ is a semidirect product $G=S_{5}\rtimes C_3$. By Theorem \ref{caratteri}, $\Delta=i\cdot65+(24-i)\cdot0+j\cdot2+(50-j)\cdot3$ with $0\leq i\leq24$ and $0\leq j\leq50$. This contradict \eqref{diffS8}.

{\bf Case $|G|=78$.} By Sylow's Third Theorem, $n_{13}=1$; by Lemma \ref{azione}, $G$ acts on the fixed points of $S_{13}$. Every nontrivial element $\sigma\in S_{13}$ generates $S_{13}$ and is either of type (A) or (B1). Hence, all nontrivial elements of $G$ either are of type (A), or act on a common triangle $T$. In the former case, $G$ contains a $2$-element of type (A), contradicting Lemma \ref{classificazione}. In the latter case, by the orbit-stabilizer theorem, the subgroup $H$ of $G$ fixing $T$ pointwise contains a $2$-element or a $3$-element. This contradicts Lemma \ref{classificazione}.

{\bf Case $|G|=80$.} By \cite[Theorem 1]{Mon}, $G$ has a characteristic $2$-subgroup $N$. By Lemma \ref{azione}, $G$ fixes the unique fixed point of $N$ on $\cH_{64}$, which is $\F_{64^2}$-rational. By Theorem \ref{fissa},
$$ 14=\frac{(64-2^w)(64-(\gcd(5,65)-1)2^v)}{2\cdot80}\,, $$
with $v+w=4$, which is impossible.

{\bf Case $|G|=84$.} By Sylow's Third Theorem, $n_7=1$. This contradicts Lemma \ref{stabilizzatore}.

{\bf Case $|G|=90$.} Since $|G|\equiv2\pmod4$, $G$ has a normal subgroup $N$ of index $2$ (see \cite[Ex. 4.3]{Ox}).
By Sylow's Third Theorem, $N$ has a characteristic $5$-subgroup $C_5$, so that $C_5$ is normal in $G$ and $n_5=1$. Also, $n_3=1$. Then $G$ is a semidirect product $G=C_5\times S_3\rtimes C_2$. By Theorem \ref{caratteri}, $\Delta=4\cdot i+40\cdot2+n_2\cdot66+(45-n_2)\cdot1$, with $i\in\{0,65\}$ and $1<n_2\mid45$. This contradicts \eqref{diffS8}.

{\bf Case $|G|=91$.} By Theorem \ref{caratteri}, $\Delta=78\cdot2+12\cdot i$ with $i\in\{0,65\}$, contradicting \eqref{diffS8}.

{\bf Case $|G|=96$.} By \cite[Theorem 1]{Mon}, $G$ has a characteristic $2$-subgroup $N$. By Lemma \ref{azione}, $G$ fixes the unique fixed point of $N$ on $\cH_{64}$, which is $\F_{64^2}$-rational. By Theorem \ref{fissa},
$$ 14=\frac{(64-2^w)(64-(\gcd(3,65)-1)2^v)}{2\cdot91}\,, $$
with $v+w=5$, which is impossible.

{\bf Case $|G|=100$.} By Sylow's Third Theorem, $n_5=1$. By Lemma \ref{pgruppo1}, the fixed points of $S_{5}$ are the vertices of a triangle $T$. By Lemma \ref{azione}, $G$ acts on $T$. By the orbit-stabilizer theorem, $G$ contains a $2$-element fixing $T$ pointwise. This contradicts Lemma \ref{classificazione}.

{\bf Case $|G|=104$.} By Sylow's Third Theorem, $n_{13}=1$. This contradicts Lemma \ref{stabilizzatore}.

{\bf Case $|G|=105$.} By Sylow's Third Theorem, $n_5\in\{1,21\}$.
All elements of a $S_5$ are of the same type, either (A) or (B1).
Then, by Theorem \ref{caratteri}, $\Delta=4i\cdot65+4(n_5-i)\cdot0+(104-4n_5)\cdot2$, with $0\leq i\leq n_5$. This contradicts \eqref{diffS8}.

{\bf Case $|G|=109$.} By Theorem \ref{caratteri}, $\Delta=108\cdot3$. This contradicts \eqref{diffS8}.

{\bf Case $|G|=111$.} By Sylow and Schur-Zassenhaus theorems, $n_{37}=1$, $n_3\in\{1,37\}$, and $G$ is a semidirect product $G=C_{37}\rtimes C_3$.
By Lemma \ref{classificazione}, $G$ has no elements of order $37\cdot3$. Hence, $n_3=37$. By Theorem \ref{caratteri}, $\Delta=36\cdot3+74\cdot2$. This contradicts \eqref{diffS8}.

{\bf Case $|G|=112$.} By \cite[Theorem 1]{Mon}, $G$ has a characteristic $2$-subgroup $N$. By Lemma \ref{azione}, $G$ fixes the unique fixed point of $N$ on $\cH_{64}$, which is $\F_{64^2}$-rational. By Theorem \ref{fissa},
$$ 14 = \frac{(64-2^w)(64-(\gcd(7,65)-1)2^v)}{2\cdot112}\,, $$
with $v+w=4$, which is a contradiction.

{\bf Case $|G|=117$.} By Sylow and Schur-Zassenhaus theorems, $G$ is a semidirect product $G=C_{13}\rtimes S_3$. Since $13$ is prime, the nontrivial elements of $C_{13}$ are of the same type (A) or (B1). By Theorem \ref{caratteri}, $\Delta=12\cdot i+104\cdot2$ with $i\in\{0,65\}$. Then $i=65$ by \eqref{diffS8}, i.e. the nontrivial elements of $C_{13}$ are homologies, with a common center $P\notin\cH_{64}$ and axis $\ell$.
By Lemma \ref{azione}, $G$ fixes $P$ and acts on $\ell$. By Lemma \ref{classificazione}, the nontrivial elements of $S_3$ are of type (B2) and fix two $\F_{64^2}$-rational points $Q,R\in\ell\cap\cH_{64}$.
Let $\cH_{64}$ have equation \eqref{normatraccia}.
Since $\PGU(3,q)$ is $2$-transitive on the $\F_{64^2}$-rational points of $\cH_{64}$, we can assume that $Q=(1,0,0)$ and $R=(0,0,1)$. Then
$ C_{13}=\{\diag(1,\lambda,1)\mid\lambda^{13}=1\}$ and $S_3=\{\diag(a^{65},a,1)\mid a^9=1\}=C_9 $;
see \cite{GSX}. Hence, $G$ is abelian and is the direct product $G=C_{13}\times C_9$. Let $\bar G\leq\PGU(3,64)$ be the group $\bar G=C_{65}\times C_9$, where $C_{65}$ is generated by $\diag(1,\bar\lambda,1)$, with $\bar\lambda$ a primitive $65$-th root of unity. Then $G$ is a normal subgroup of $\bar G$ of index $5$, so that $\bar G/G\leq\aut(\cH_{64}/G)$ has order $5$. Also, $\bar G/G$ has two $\F_{8}$-rational fixed points on $\cH_{64}/G$, namely the points lying under $Q$ and $R$. This is inconsistent with the structure of the automorphism group of $\cS_8$. In fact, by \cite[Theorem A.12]{HKT} (see also \cite[Remark (2.2)]{GKT}), any subgroup of $\aut(\cS_8)$ of order $5$ is a Singer group acting semiregularly on the $\F_{8}$-rational points of $\cS_8$.

{\bf Case $|G|=120$.} By \cite[Ex. 8.19]{Ox}, either $n_5=1$, or $G$ has a normal $2$-subgroup, or $G$ is isomorphic to the symmetric group $Sym(5)$.
The case $n_5=1$ is impossible by Lemma \ref{stabilizzatore}. Hence, $n_5=6$.
Suppose that $G$ has a normal $2$-subgroup $N$. By Lemma \ref{azione}, $G$ fixes the unique fixed point of $N$ on $\cH_{64}$. Then any $5$-element of $G$ is of type (A) by Lemma \ref{classificazione}. By Theorem \ref{caratteri}, $\Delta\geq24\cdot65$; this contradicts \eqref{diffS8}.
Suppose that $G\cong Sym(5)$. Then $G$ contains $25$ involutions. By Theorem \ref{caratteri}, $\Delta\geq25\cdot66$. This contradicts \eqref{diffS8}.

{\bf Case $|G|=126$.} Since $|G|\equiv2\pmod4$, $G$ has a normal subgroup $N$ of index $2$. Then $G$ is a semidirect product $G=N\rtimes C_2$. By Theorem \ref{caratteri}, $\Delta=62\cdot2+n_2\cdot66+(63-n_2)\cdot1$. This contradicts \eqref{diffS8}.

{\bf Case $|G|=128$.} By Theorem \ref{Mit}, $G$ fixes an $\F_{64^2}$-rational point of $\cH_{64}$. Then, by Theorem \ref{fissa2}, $14=2^{6-v-1}(2^{6-w}-1)$ with $0\leq v,w\leq5$. Hence, $v=4$, $w=3$. By theorem \ref{fissa2}, $G$ contains exactly $2^3-1$ involutions. By Theorem \ref{caratteri}, $\Delta=7\cdot66+120\cdot1$. This contradicts \eqref{diffS8}.

{\bf Case $|G|=130$.} By Sylow's Third Theorem, $n_{13}=1$, $n_5\in\{1,26\}$, and $n_2\in\{1,5,13,65\}$.
By \eqref{diffS8}, $\Delta=650$. Hence, by Theorem \ref{caratteri}, the nontrivial elements of $S_{13}$ are of type (B1). We remark that if $x$ is an element of type (C) normalizing an element $y$ of type (A) or (B1), then the element $yx$ is of type (E). If $n_5=1$, then $G$ is a semidirect product $G=C_{65}\rtimes C_2$; hence, $n_2=1$ by the above remark.
If $n_5=26$, then $G$ contains $12$ elements of order $13$, $4\cdot26$ elements of order $5$, and 12 elements of type (E) by the above remark. Hence, $n_2=1$.
Therefore $G$ contains a unique involution $\sigma$. By Lemma \ref{azione}, $S_{13}$ fixes the unique fixed point of $\sigma$ on $\cH_{64}$. This contradicts Lemma \ref{classificazione}.


{\bf Case $|G|=140$.} By Sylow's Third Theorem, $n_7=1$. This contradicts Lemma \ref{stabilizzatore}.

{\bf Case $|G|=144$.} By Theorem \ref{caratteri}, $\Delta=i\cdot66+j\cdot1+k\cdot2$ with $i+j+k=143$. Here, $i$ is the number of involutions in $G$, $j$ is the number of elements of order $6$ or $18$ in $G$, and $k$ is the number of elements of order $3$, $9$, or $4$ in $G$.
Suppose $i=1$. Then, by Lemma \ref{azione}, $G$ fixes the unique fixed point of the involution on $\cH_{64}$, which is $\F_{64^2}$-rational. By Theorem \ref{fissa},
$$ 14=\frac{(64-2^w)(64-(\gcd(9,65)-1)2^v)}{2\cdot144}\,, $$
with $v+w=4$, hence $w=0$. By Theorem \ref{fissa}, $G$ has no involutions, which is impossible.
Then $i\geq2$ and thus by \eqref{diffS8}, we have $i=2$ and $k=13$.
This implies that $G$ contains $2$ involutions and at most $13$ elements of order $4$. Hence, $G$ has a unique Sylow $2$-subgroup $S_{2}$.
Then, by Lemma \ref{azione}, $G$ fixes the unique fixed point of $S_{2}$ on $\cH_{64}$ and as before, it leads to a contradiction by Theorem \ref{fissa}.

{\bf Case $|G|=148$.} By Theorem \ref{Mit}, $|G|$ does not divide three times the order of any maximal subgroup of $\PSU(3,64)$. Hence, $G$ is not contained in any maximal subgroup of $\PGU(3,64)$, a contradiction.

{\bf Case $|G|=150$.} By Lemma \ref{pgruppo1}, $G$ contains $8$ elements of type (A). Hence, by Theorem \ref{caratteri}, $\Delta\geq8\cdot65$. This contradicts \eqref{diffS8}.


This completes the proof of Theorem \ref{teoremaS8}.

It may be noticed in the above proof that the hypothesis $g=14$ together with the $\F_{64^2}$-maximality of $\cS_8$ were sufficient to get a contradiction for $|G|\neq 117$.
Instead, a group $G$ of order $117$ with the required ramification exists, and we gave an explicit construction.
%
Such a group $G$ is uniquely determined up to conjugation.
Using MAGMA \cite{MAGMA}, we found a plane model of $\cH_{64}/G$ over $\mathbb F_2$, as well as a non-singular model of $\cH_{64}/G$ in $\PG(13,2)$.
\begin{proposition}\label{H/G}
Let $\cX$ be an $\F_{64^2}$-maximal curve of genus $14$. If $\cX$ is Galois covered by $\cH_{64}$ then $\cX\cong\cH_{64}/G$ where $G$ is a cyclic group $G\leq\PGU(3,64)$ of order $117$, and
a plane model of $\cX$ over $\mathbb F_2$ is the (singular) plane curve
$$\quad X^7Y^5+X+Y^5=0, $$ while
a nonsingular model in $\mathbb P^{13}$ of $\cX$ over $\mathbb F_2$ is the image of $\cX$ under the morphism
$$\varphi:\cX\rightarrow\mathbb P^{13},\quad (x,y,1)\mapsto(x,y,xy,x^2y,y^2,xy^2,x^2y^2,x^3y^2,y^3,xy^3,x^2y^3,x^3y^3,x^4y^3,1).$$
\end{proposition}

\section{Proof of Theorem \ref{teoremaR3}}\label{sezioneR3}

By absurd, let $\cR_3\cong\cH_{27}/G$ for $G\leq\PGU(3,27)$. The order of $\PGU(3,27)$ is equal to $2^5\cdot3^9\cdot7^2\cdot13\cdot19\cdot37$. From the Riemann-Hurwitz formula,
$$ 12<\frac{|\cH_{27}(\F_{27^2})|}{|\cR_3(\F_{27^2})|}\leq|G|\leq\frac{2g(\cH_{27})-2}{2g(\cR_3)-2}\leq 25. $$
Since $|G|$ divides $|\PGU(3,27)|$,
$$|G| \in \{13,14,16,18,19,21,24\}.$$
The different divisor has degree
\begin{equation}\label{diffR3}
\Delta=(2g(\cH_{27})-2)-|G|(2g(\cR_3)-2) = 700-28\cdot|G|.
\end{equation}

{\bf Case $|G|=13$.} By Theorem \ref{caratteri}, $\Delta=12\cdot2$. This contradicts \eqref{diffR3}.

{\bf Case $|G|=14$.} By Sylow and Schur-Zassenhaus theorems, $G$ is a semidirect product $G=C_7\rtimes C_2$. All nontrivial elements of $C_7$ are of the same type, which is either (A) or (B1) by Lemma \ref{classificazione}.
Therefore, by Theorem \ref{caratteri}, $\Delta=6\cdot i+7\cdot28$, with $i\in\{0,28\}$. This contradicts \eqref{diffR3}.

{\bf Case $|G|=16$.} $\PGU(3,27)$ has just three conjugacy classes of subgroups of order $16$, which are isomorphic either to the Iwasawa group $M_{16}=\langle x,y\mid x^8=y^2=1,yxy^{-1}=x^5 \rangle$, or to the direct product $C_4\times C_4$, or to the central product $D_8\circ C_4=\langle \alpha,\beta,\gamma\mid \alpha^4=\beta^2=1,\beta\alpha\beta^{-1}=\alpha^{-1},\alpha^2=\gamma^2,\alpha\gamma=\gamma\alpha,\beta\gamma=\gamma\beta \rangle$.

Suppose $G\cong M_{16}$. By MAGMA computation, the normalizer $N$ of $G$ in $\PGU(3,27)$ has order $224$, and the quotient group $N/G\leq\aut(\cH_{27}/G)$ is a cyclic group of order $14$. On the other hand, the subgroups of $Ree(3)\cong{\rm P\Gamma L}(2,8)$ of order $14$ are not abelian, a contradiction.

Suppose $G\cong C_4\times C_4$. By MAGMA computation, the normalizer $N$ of $G$ in $\PGU(3,27)$ has order $4704$. Hence, the group $N/G\leq\aut(\cH_{27}/G)$ has order $294$, which does not divide the order of $Ree(3)$. This contradicts $\cH_{27}/G\cong\cR_3$.

Suppose $G\cong D_8\circ C_4$. By MAGMA computation, the normalizer $N$ of $G$ in $\PGU(3,27)$ has order $672$, and the group $N/G\leq\aut(\cH_{27}/G)$ is isomorphic to $C_{21}\rtimes C_2$. On the other hand, the subgroups of $Ree(3)$ of order $42$ have no cyclic subgroups of order $21$, a contradiction.

{\bf Case $|G|=18$.} By Sylow's Third Theorem, $n_3=1$.
By \cite[Theorem 11.74]{HKT}, $S_3$ has a unique fixed point $P$ on $\cH_{27}$, which is $\F_{27^2}$-rational.
By Lemma \ref{azione}, $G$ fixes $P$. Then, by Theorem \ref{fissa},
$$ 15=\frac{(27-3^w)(27-(\gcd(2,28)-1)3^v)}{2\cdot18}\,, $$
with $v+w=2$, which is impossible.

{\bf Case $|G|=19$.} By Theorem \ref{caratteri}, $\Delta=18\cdot3$. This contradicts \eqref{diffR3}.

{\bf Case $|G|=21$.} By Sylow and Schur-Zassenhaus theorems, $G$ is a semidirect product $G=C_7\rtimes C_3$. All nontrivial elements of $C_7$ are of the same type, which is either (A) or (B1) by Lemma \ref{classificazione}.
Thus, by Theorem \ref{caratteri}, $\Delta=6\cdot i+2n_3\cdot29+(14-2n_3)\cdot1$, with $i\in\{0,28\}$. This contradicts \eqref{diffR3}.

{\bf Case $|G|=24$.} Since $3\mid|G|$, we have $\Delta\geq29$ by Theorem \ref{caratteri}. This contradicts \eqref{diffR3}.

This completes the proof of Theorem \ref{teoremaR3}.

It may be noticed in the above proof that the hypothesis $g=15$ together with the $\F_{27^2}$-maximality of $\cR_3$ ruled out all cases but $|G|=16$. For this exception, three cases are treated separately.
\begin{itemize}
\item $G\cong M_{16}$. Then $G$ contains $3$ involutions, $4$ elements of order $4$, and $8$ elements of order $8$. By Theorem \ref{caratteri}, the quotient curve $\cH_{27}/G$ has genus $18$.
\item $G\cong C_4\times C_4$. By the Riemann-Hurwitz formula, $\cH_{27}/G$ has genus $15$. Also, $G$ has $9$ elements of type (A) and $6$ elements of type (B1). Hence, $G$ fixes the vertices of a triangle $T$. Let $\cH_{27}$ have equation \eqref{fermat}. Up to conjugation, $T$ is the fundamental triangle and $G=\{\diag(\lambda,\mu,1)\mid\lambda^4=\mu^4=1\}$. Therefore a (singular) plane model of $\cH_{27}/G$ is $X^{7}+Y^{7}+1=0$.
\item  $G\cong D_8\circ C_4$. By the Riemann-Hurwitz formula, $\cH_{27}/G$ has genus $15$. Also, $G$ contains $9$ elements of type (A) and $6$ elements of type (B1). In particular, the non-central involutions of $D_8$ are non-commuting elements of type (A). Thus, the generator $y$ of the center $C_4$ is not of type (B1). Hence, $y$ is of type (A).
Let $\cH_{27}$ have equation \eqref{fermat}. Up to conjugation, the generators of $G$ are
$\alpha:(X,Y,T)\mapsto(Y,-X,T)$, $\beta=\diag(1,-1,1)$, and $\gamma=\diag(\lambda,\lambda,1)$, where $\lambda^2=-1$. A plane model of $\cH_{27}/G$ is obtained by MAGMA computation, as follows.
\end{itemize}

\begin{proposition}\label{H27/G}
Let $\cX$ be an $\F_{27^2}$-maximal curve of genus $15$. If $\cX$ is Galois covered by $\cH_{27}$ then $\cX\cong\cH_{27}/G$ where $G \le \PGU(3,27)$ has order $16$, and one of the following cases occurs.
\begin{itemize}
\item $G\cong C_4\times C_4$ and a plane model for $\cX$ is given by the affine equation $$X^7+Y^7+1=0\,.$$
\item $G\cong D_8 \circ C_4$ and a plane model for $\cX$ is given by the affine equation
$$ X^{28} + X^{27} + X^{26} + 2X^{23} + 2X^{22} + X^{21} + 2X^{12}Y^{14} +
        X^{10}Y^{14} + 2X^{7}Y^{14} + Y^{28}=0. $$
\end{itemize}
\end{proposition}

\section{Galois subcovers of $\cH_{27}$}\label{sezioneSpettro}

Theorem \ref{spettro} shows the complete spectrum of genera of Galois subcovers of $\cH_{27}$, consisting of integers $g$ which are the genera of a quotient curve $\cH_{27}/G$ with $G$ ranging on the set of all subgroups of $PGU(3,27)$.
\begin{theorem}\label{spettro}
The spectrum of genera of Galois subcovers of $\cH_{27}$ is
$$ \Sigma_{27} = \{0,1,3,4,5,6,7,9,10,12,13,15,16,17,18,19,24,25,$$ $$\hspace{2 cm} 26,27,36,39,43,51,52,78,85,108,117,169,351\}. $$
\end{theorem}
The proof relies on the results of Section \ref{preliminari}.  A case-by-case analysis of all integers $g$ with $1<g\leq g(\cH_{27})$ is combined with
\begin{equation}\label{limitiamo} \frac{19684}{730+54g}= \frac{|\cH_{27}(\F_{27^2})|}{|\cH_{27}/G(\F_{27^2})|} \leq|G|\leq \frac{2g(\cH_{27})-2}{2g(\cH_{27}/G)-2} = \frac{700}{2g-2}\,, \end{equation}
which bounds the order of a putative group $G\leq\PGU(3,27)$ such that $\cH_{27}/G$ has genus $g$. This leads us to
look inside the structure of the groups $G$ satisfying \eqref{limitiamo} and compute the genus of $\cH_{27}/G$, for $g>1$. These results are summarized in Theorem \ref{spettro}. For each $g>1$ in $\Sigma_{27}$, Tables {\rm \ref{tabella1}} and {\rm \ref{tabella2}} provide a classification of the groups $G$ for which $\cH_{27}/G$ has genus $g$.

\begin{center}
\begin{table}[htbp]
\begin{small}
\caption{Quotient curves $\cH_{27}/G$ of genus $g\geq 17$}\label{tabella1}
\begin{tabular}{|c|c|c|}
\hline \bf $g$ & $|G|$ & \textrm{structure of $G$} \\
\hline\hline 351 & 1 & \textrm{trivial group.}\\
\hline 169 & 2 & $G=C_2=\langle\sigma\rangle$, $\sigma$ of type (A).\\
\hline 117 & 3 & $G=C_3=\langle\sigma\rangle$, $\sigma$ of type (D).\\
\hline 108 & 3 & $G=C_3=\langle\sigma\rangle$, $\sigma$ of type (C).\\
\hline 85 & 4 & $G=C_4=\langle\sigma\rangle$, $\sigma$ of type (B1).\\
\hline 78 & 4 & $G=C_4=\langle\sigma\rangle$, $\sigma$ of type (A).\\
\hline 52 & 6 & $G=C_6=\langle\sigma\rangle$, $\sigma$ of type (E).\\
  &  & $G=Sym(3)=\langle\alpha\rangle\rtimes\langle\beta\rangle$, $\alpha$ of type (C), $\beta$ of type (A).\\
\hline 51 & 7 & $G=C_7=\langle\sigma\rangle$, $\sigma$ of type (B1).\\
\hline 43 & 8 & $G= Q_8$ quaternion group, 1 element of type (A), 6 elements of type (B1).\\
\hline 39 & 7 & $G=C_7=\langle\sigma\rangle$, $\sigma$ of type (A).\\
 & 8 & $G=C_8=\langle\sigma\rangle$, $\sigma$ of type (B2).\\
 & 9 & $G=C_3\times C_3=\langle\alpha\rangle\times\langle\beta\rangle$, $\alpha$ and $\beta$ of type (D).\\
\hline 36 & 8 & $G=C_4\times C_2=\langle\alpha\rangle\times\langle\beta\rangle$, $\alpha$ of type (B1), $\beta$ of type (A).\\
 & 8 & $G=D_8$ dihedral group, involutions of type (A), 2 elements of type (B1).\\
 & 9 & $G=C_3\times C_3=\langle\alpha\rangle\times\langle\beta\rangle$, $\alpha$ of type (C), $\beta$ of type (D).\\
\hline 27 & 9 & $G=C_3\times C_3=\langle\alpha\rangle\times\langle\beta\rangle$, $\alpha$ and $\beta$ of type (C).\\
 & 13 & $G=C_{13}=\langle\sigma\rangle$, $\sigma$ of type (B2).\\
\hline 26 & 12 & $G= Alt(4)$, involutions of type (A), other elements of type (D).\\
\hline 25 & 14 & $G=C_{14}=\langle\sigma\rangle$, $\sigma$ of type (B1).\\
\hline 24 & 12 & $G=C_{12}=\langle\sigma\rangle$, $\sigma$ of type (E).\\
\hline 19 & 14 & $G=C_{14}=\langle\sigma\rangle$, $\sigma$ of type (B1), $\sigma^2$ of type (A).\\
\hline 18 & 16 & $G= M_{16}$, 5 elements of type (A),\\ & & 2 elements of type (B1), 8 elements of type (B2).\\
\hline 18 & 19 & $G=C_{19}=\langle\sigma\rangle$, $\sigma$ of type (B3).\\
\hline 17 & 21 & $G=C_7\rtimes C_3=\langle\alpha\rangle\rtimes\langle\beta\rangle$, $\alpha$ of type (B1), $\beta$ of type (B2).\\
\hline
\end{tabular}
\end{small}
\end{table}
\end{center}

\begin{center}
\begin{table}[htbp]
\begin{small}
\caption{Quotient curves $\cH_{27}/G$ of genus $3 \leq g \leq 16$}\label{tabella2}
\begin{tabular}{|c|c|c|}
\hline \bf $g$ & $|G|$ & \textrm{structure of $G$} \\
\hline\hline
 16 & 18 & $G=C_3\times (C_3\rtimes C_2)=\langle\alpha\rangle\times(\langle\beta\rangle\rtimes\langle\gamma\rangle)$,\\ & & $\alpha$ of type (C), $\beta$ of type (D), $\gamma$ of type (A).\\
\hline 15 & 16 & $G=C_4\times C_4=\langle\alpha\rangle\times\langle\beta\rangle$, $\alpha$ and $\beta$ of type (A).\\
 & 16 & $G=D_8\circ C_4=(\langle\alpha\rangle\rtimes\langle\beta\rangle)\circ\langle\gamma\rangle$, $\alpha$ of type (B1), $\beta$ and $\gamma$ of type (A).\\
 13 & 14 & $G=C_{14}=\langle\sigma\rangle$, $\sigma$ of type (A).\\
 & 18 & $G=C_3\times (C_3\rtimes C_2)=\langle\alpha\rangle\times(\langle\beta\rangle\rtimes\langle\gamma\rangle)$, $\alpha$ and $\beta$ of type (D), $\gamma$ of type (A).\\
 & 18 & $G=C_3\times (C_3\rtimes C_2)=\langle\alpha\rangle\times(\langle\beta\rangle\rtimes\langle\gamma\rangle)$, $\alpha$ and $\beta$ of type (C), $\gamma$ of type (A).\\
 & 26 & $G=C_{26}=\langle\sigma\rangle$, $\sigma$ of type (E).\\
 & 27 & $G=(C_3\times C_3)\rtimes C_3=(\langle\alpha\rangle\times\langle\beta\rangle)\rtimes\langle\gamma\rangle$, $\alpha$, $\beta$, $\gamma$ of type (D).\\
 & 28 & $G=C_{28}=\langle\sigma\rangle$, $\sigma$ of type (B1), 1 element of type (A).\\
\hline 12 & 21 & $G=C_{21}=\langle\sigma\rangle$, $\sigma$ of type (E).\\
 & 24 & $G=C_3\RT C_8=\la\alpha\ra\rtimes\la\beta\ra$, $\alpha$ of type (C), $\beta$ of type (B2).\\
 & 27 & $G=C_3\times (C_3\times C_3)=\langle\alpha\rangle\times(\langle\beta\rangle\times\langle\gamma\rangle)$, $\alpha$ of type (C), $\beta$ and $\gamma$ of type (D).\\
 & 28 & $G=C_{28}=\langle\sigma\rangle$, $\sigma$ of type (B1), 3 elements of type (A).\\
 & 28 & $G=C_{14}\times C_2=\langle\alpha\rangle\times\langle\beta\rangle$, $\alpha$ of type (B1), $\beta$ of type (A), 3 elements of type (A).\\
\hline 10 & 24 & $G\cong{\rm SL}(2,3)$, 1 element of type (A), 6 elements of type (B1),\\ & & 8 elements of type (C), 8 elements of type (E).\\
\hline 9 & 37 & $G=C_{37}=\langle\sigma\rangle$, $\sigma$ of type (B3).\\
\hline 7 & 26 & $G=C_{13}\rtimes C_2=\langle\alpha\rangle\rtimes\langle\beta\rangle$, $\alpha$ of type (B2), $\beta$ of type (A).\\
 & 28 & $G=C_{28}=\langle\sigma\rangle$, $\sigma$ of type (B1), 14 elements of type (A).\\
 & 52 & $G=C_{13}\RT C_4=\la\alpha\ra\RT\la\beta\ra$, $\alpha$ of type (B2), $\beta$ of type (B1).\\
\hline 6 & 28 & $G=C_{14}\times C_2=\langle\alpha\rangle\times\langle\beta\rangle$, $\alpha$ and $\beta$ of type (A), 15 elements of type (A).\\
 & 32 & $G=C_4\wr C_2 = \langle\alpha\rangle\wr\langle\beta\rangle$ wreath product,\\ & & 13 elements of type (A), 10 elements of type (B1), 8 elements of type (B2).\\
 & 52 & $G=C_{52}=\la\sigma\ra$, $\sigma$ of type (B2), 3 elements of type (A).\\
 & 57 & $G=C_{19}\RT C_3 = \la\alpha\ra\RT\la\beta\ra$, $\alpha$ of type (B3), $\beta$ of type (D).\\
\hline 5 & 48 & $G=(C_4\times C_4)\rtimes C_3=(\la\alpha\ra\T\la\beta\ra)\RT\la\gamma\ra$, $\alpha$ and $\beta$ of type (A), $\gamma$ of type (D).\\
\hline 4 & 42 & $G=C_{42}=\langle\sigma\rangle$, $\sigma$ of type (E).\\
 & 48 & $G=(D_8\circ C_4)\RT\la\sigma\ra$, $\sigma$ of type (C).\\
 & 54 & $G=(C_3\T C_3 \RT C_3)\RT\la\sigma\ra$, $\sigma$ of type (A).\\
 & 56 & $G=Q_8\T \la\sigma\ra$, $\sigma$ of type (A).\\
 & 72 & $G=C_4 \times C_2 \rtimes(\la\alpha\ra\times \la\beta\ra)$, $\alpha$ and $\beta$ of type (D).\\
 & 81 & $G= C_3\T C_3\T C_3\T C_3 = \la\alpha\ra\T \la\beta\ra\T \la\gamma\ra\T \la\delta\ra$, $\alpha$ of type (C), $\beta,\gamma,\delta$ of type (D).\\
\hline 3 & 49 & $G=C_7\times C_7=\langle\alpha\rangle\times\langle\beta\rangle$, $\alpha$ and $\beta$ of type (A).\\
 & 56 & $G=\la\sigma\ra\RT D_8$, $\sigma$ of type (A).\\
 & 63 & $G=C_7\T C_3 \T C_3 = \la\alpha\ra\T\la\beta\ra\T\la\gamma\ra$, $\alpha$ of type (A), $\beta$ and $\gamma$ of type (C).\\
 & 72 & $G=C_3 \T C_3 \RT C_8 = \la\alpha\ra\T\la\beta\ra\RT\la\gamma\ra$, $\alpha$ and $\beta$ of type (C), $\gamma$ of type (B2).\\
 & 81 & $G= C_3\T C_3\T C_3\T C_3 = \la\alpha\ra\T \la\beta\ra\T \la\gamma\ra\T \la\delta\ra$, $\alpha,\beta$ of type (C), $\gamma,\delta$ of type (D).\\
 & 91 & $G=C_{91}=\la\sigma\ra$, $\sigma$ of type (B2).\\
 & 104 & $G=C_{13}\RT C_8=\la\alpha\ra\RT\la\beta\ra$, $\alpha$ and $\beta$ of type (B2), or $G=C_{104}=\la\sigma\ra$, $\sigma$ of type (B2).\\
 & 111 & $G=C_{37}\RT C_3=\la\alpha\ra\RT\la\beta\ra$, $\alpha$ of type (B3), $\beta$ of type (D).\\
 & 112 & $G=C_7\T C_4 \T C_4 = \la\alpha\ra\T\la\beta\ra\T\la\gamma\ra$, $\alpha$ of type (B1), $\beta$ and $\gamma$ of type (A).\\
\hline
\end{tabular}
\end{small}
\end{table}
\end{center}

Theorem \ref{spettro} shows that some quotient curves of $\cR_3$ happen not be Galois subcovers of $\cH_{27}$. A partial list of them is given in the following proposition.
\begin{corollary}\label{noncoperte}
The quotient curves $\cR_3/G_1$, $\cR_3/G_2$, and $\cR_3/G_3$ are not Galois subcovers of $\cH_{27}$ for the groups $G_1,G_2,G_3$ defined as follows.
\begin{itemize}
\item The maximal subgroups $G_1\leq Ree(3)$ of order $24$ centralizing an involution $\sigma\in Ree(3)$, which are isomorphic to $\la\sigma\ra\times Alt(4)$.
\item The groups $G_2\leq Ree(3)$ of order $6$ which are isomorphic to $Sym(3)$.
\item  The cyclic groups $G_3\leq Ree(3)$ of order $6$.
\end{itemize}
\end{corollary}
\begin{proof}
 From previous work of \c{C}ak\c{c}ak and \"Ozbudak, each of the quotient curves $\cR_3/G_1$, $\cR_3/G_2$, and $\cR_3/G_3$ has genus $2$; see \cite[Sec. 4.1.1, p. 150]{CO} for $\cR_3/G_1$, \cite[Sec. 4.2, pp. 163-164]{CO} for $\cR_3/G_2$, and \cite[Sec. 4.4, pp. 171]{CO} for $\cR_3/G_3$. On the other hand,
Theorem \ref{spettro} shows that no $\F_{27^2}$-maximal curve of genus $2$ is a Galois subcover of $\cH_{27}$.
\end{proof}

\section{Acknowledgements}
This research was supported by  the Italian Ministry MIUR, Strutture Geometriche, Combinatoria e loro Applicazioni, PRIN 2012 prot. 2012XZE22K, and by GNSAGA of the
Italian INdAM.

\end{document}